\documentclass[a4paper,12pt]{amsart}

\usepackage{enumitem}

\makeatletter
\@namedef{subjclassname@1991}{2020  Mathematics Subject Classification}
\makeatother

\usepackage[utf8]{inputenc}
\usepackage{amstext}
\usepackage{amsfonts}
\usepackage{amsthm}
\usepackage{textcomp}
\usepackage{amssymb}
\usepackage{amscd}
\usepackage{amsmath}
\usepackage{xcolor}

\makeatletter
\newtheorem*{rep@theorem}{\rep@title}
\newcommand{\newreptheorem}[2]{%
\newenvironment{rep#1}[1]{%
 \def\rep@title{#2 \ref{##1}}%
 \begin{rep@theorem}}%
 {\end{rep@theorem}}}
\makeatother

\newtheorem{Theorem}{Theorem}[section]
\newreptheorem{Theorem}{Theorem}
\newreptheorem{Corollary}{Corollary}
\newtheorem{Lemma}[Theorem]{Lemma}
\newtheorem{Proposition}[Theorem]{Proposition}
\newtheorem{Corollary}[Theorem]{Corollary}

\newtheorem{Definition}[Theorem]{Definition}

\newtheorem{Remark}[Theorem]{Remark}

\newcommand{\Z}{\mathbb{Z}}

\newcommand{\reg}{\textnormal{reg}}

\newcommand{\spn}{\operatorname{span}}

\newcommand{\ol}{\overline}

\begin{document}

\title{A note on the regular ideals of Leavitt path algebras}

\author[D. Gon\c{c}alves]{Daniel Gon\c{c}alves}
\address{Daniel Gon\c{c}alves - Departamento de Matem\'{a}tica - UFSC - Florian\'{o}polis - SC, Brazil}
\email{daemig@gmail.com}
\thanks{D. Gon\c{c}alves was partially supported by Conselho Nacional de Desenvolvimento Cient\'ifico e Tecnol\'ogico (CNPq) and Capes-PrInt grant number 88881.310538/2018-01 - Brazil.}

\author[D. Royer]{Danilo Royer}
\address{Danilo Royer - Departamento de Matem\'{a}tica - UFSC - Florian\'{o}polis - SC, Brazil}
\email{daniloroyer@gmail.com}




\begin{abstract}
 We show that, for an arbitrary graph, a regular ideal of the associated Leavitt path algebra is also graded. As a consequence, for a row-finite graph, we obtain that the quotient of the associated Leavitt path by a regular ideal is again a Leavitt path algebra and that Condition~(L) is preserved by quotients by regular ideals. Furthermore, we describe the vertex set of a regular ideal and make a comparison between the theory of regular ideals in Leavitt path algebras and in graph C*-algebras.

\end{abstract}

\maketitle

\section{Introduction}

Leavitt path algebras arose as algebraization of graph C*-algebras and are the subject of intense research. Although tempting, we refrain from presenting an exposition of the beginning and development of the field and refer the reader to the excellent works \cite{Gene, book}.

Our interest in this note is to explore regular ideals (in the sense of Hamana, \cite{Hamana}) of Leavitt path algebras. We will see that the algebraic theory of regular ideals has resemblances with the theory of regular ideals in graph C*-algebras, as we show algebraic analogues of the results proved for graph C*-algebras by Brown, Fuller, Pitts, and Reznikoff in \cite{galera}. We will also describe some differences between the analytical and algebraic settings. In particular, we will prove that every regular ideal in a Leavitt path algebras is graded, while the equivalent statement is not true in general for graph C*-algebras, see Remark~\ref{ondesol}. 

In the analytical setting, regular ideals of graph C*-algebras are connected with gauge-invariant ideals. The algebraic analogue of a gauge-invariant ideal is a $\Z$-graded ideal (or simply called graded ideal). 
Graded ideals and structures play a key role in the theory of Leavitt path algebras. We mention a few examples, as \cite{ABHS}, where the graded structure of Steinberg algebras is used to deduce that diagonal preserving ring isomorphism of Leavitt path algebras induces C*-isomorphism of C*-algebras for graphs that satisfy Condition~(L),
\cite{AHHS}, where it is deduced that the lattice of order-ideals in the $K_0$-group of the Leavitt path algebra is isomorphic to the lattice of graded ideals of the algebra, \cite{HR}, where graded irreducible representations of Leavitt path algebras are studied, \cite{HRS}, where it is shown that for a finite graph, graded regular graded self-injective Leavitt path algebras are of graded type I, \cite{CHR}, where a complete graphical characterization of strongly graded Leavitt path algebras is given, \cite{AMMS}, where the graded uniqueness theorem is proved via socle theory, and \cite{Pardo}, where it is shown that there is a natural isomorphism between the lattice of graded ideals of a Leavitt path algebras and the monoid of isomorphism classes of finitely generated projective modules. 

Next, we provide a more detailed description of what we will prove in this note. After recalling the relevant preliminary concepts on Leavitt path algebra and graded ideals in Section 2, we start Section~3 with the definition of a regular ideal, see Definition~\ref{def: reg}. We then proceed to prove our main theorem, Theorem~\ref{avai ganhou do vasco}, from where we obtain, as an immediate Corollary, that all regular ideals in a Leavitt path algebra are graded, see Corollary~\ref{regular are graded}. We describe the vertex set of a regular ideal in Proposition~\ref{tomaraqdc} and describe relations between regular ideals and cycles in a Leavitt path algebra in Proposition~{bijecao}. We then proceed to study the Leavitt path algebra associated with a quotient graph: Given a Leavitt path algebra, it is known that its quotient by a graded ideal is again a Leavitt path algebra (of a different graph, see \cite[Theorem~2.4.15]{book}). This is the case, for example, of any quotient of an exchange Leavitt path algebra (that is, a Leavitt path algebra associated to a graph that satisfies Condition~(K)), since for these algebras all of its ideals are graded (see \cite[Proposition~2.9.9]{book}). In Proposition~\ref{lagosta} we show that, for row-finite graphs, Condition~(L) is preserved by quotients by regular ideals, and show that, for graphs that are row-finite and satisfy Conditon~(L), the quotient of the associated algebra by a regular ideal is again a Leavitt path algebra. In Remark~\ref{ondesol} we point to some differences between the analytical and algebraic theory and, finally, we finish the paper with some remarks about maximal regular ideals.

\section{Preliminaries}

In this section, we recall the key concepts that we will need in this note. We follow the notation in \cite{book, Tomforde}.

\subsection{Leavitt path algebras}

A graph $E := (E^0, E^1, r, s)$ consists of a countable set of vertices $E^0$, a countable set of edges $E^1$, and maps $r: E^1 \to E^0$ and $s:E^1 \to E^0$ identifying the range and source of each edge. A graph $E$ is \emph{row-finite} if $s^{-1}(v)$ is finite for each vertex $v$.  If $E$ is a graph, a \emph{nontrivial path} is a sequence $\alpha := e_1 e_2 \ldots e_n$ of edges with $r(e_i) = s(e_{i+1})$ for $1 \leq i \leq n-1$.  We say the path $\alpha$ has \emph{length} $| \alpha| :=n$, and we let $E^n$ denote the set of paths of length $n$.  We consider the vertices in $E^0$ to be trivial paths of length zero.  We also let $E^* := \bigcup_{n=0}^\infty E^n$ denote the paths of finite length. For a path $\alpha$, we let
$\alpha^0:=\{r(e_i)\}\cup\{s(e_i)\}$ be the set of vertices
in the path $\alpha$. A path $\alpha=e_1\ldots e_n$ is a \emph{cycle} if $r(e_n)=s(e_1)$ and $s(e_i)\neq s(e_j)$ for each $i\neq j$. An \emph{exit} for $\alpha$ is an edge $e$ such that $s(e)=s(e_i)$ for some $i$, but $e\neq e_i$. A graph is said to satisfy \emph{Condition~(L)} if every cycle has an exit. The set $E_{reg}^0$ of \emph{regular vertices} is the set of all the vertices $v$ such that $s^{-1}(v)$ is nonempty and finite. If $|s^{-1}(v)|=\infty$ then $v$ is called an \emph{infinite emitter}. Let $H \subseteq E^0$.
The set $H$ is \emph{hereditary} if whenever $\alpha\in E^*$ satisfies
$s(\alpha)\in H$, then $r(\alpha)\in H$.
The set $H$ is \emph{saturated} if  for all $v\in E_{reg}^0$, $r(s^{-1}(v))  \subseteq H$ implies $v\in H$. 
A vertex $w$ is called a \emph{breaking vertex} of $H$ if $w\in E^0\setminus H$ is an infinite emitter and $1\leq |s^{-1}(w)\cap r^{-1}(E^0\setminus H)|<\infty$. We denote by $B_H$ the set of all the breaking vertices of $H$. For each $v\in B_H$ we define $v^H=v-\sum\limits_{s(e)=v; \,\,r(e)\notin H }ee^*$.

The Leavitt path algebra associated with a graph is defined as follows:

\begin{Definition} \label{roncando}
Let $E$ be a directed graph, and let $K$ be a field.  The \emph{Leavitt path algebra of $E$ with coefficients in $K$}, denoted $L_K(E)$,  is the universal $K$-algebra generated by a set $\{v : v \in E^0 \}$ of pairwise orthogonal idempotents, together with a set $\{e, e^* : e \in E^1\}$ of elements satisfying
\begin{enumerate}
\item $s(e)e = er(e) =e$ for all $e \in E^1$,
\item $r(e)e^* = e^* s(e) = e^*$ for all $e \in E^1$,
\item $e^*f = \delta_{e,f} \, r(e)$ for all $e, f \in E^1$, and
\item $v = \displaystyle \sum_{\{e \in E^1 : s(e) = v \}} ee^*$ whenever $v \in E^0_\reg$.
\end{enumerate}
\end{Definition}
\begin{Remark}
The Leavitt path algebra associated with a graph can also be constructed via partial skew group ring theory or via groupoid (Steinberg) algebras theory, see \cite{Canto, CMMS, dd2, dy, Ri}.
\end{Remark}

\begin{Definition}
Let $E$ be a graph
and $I$ be an ideal in the Leavitt path algebra $L_K(E)$.
Define
\[ H(I) := \{ v \in E^0 \colon v \in I \}. \]
\end{Definition}

If $H \subseteq E^0$ is hereditary, then we denote by $I(H)$ the ideal in $L_K(E)$ generated by $H$, that is, 
\begin{equation}\label{jacare}
 I(H) :=\spn\{\gamma \lambda^*: \gamma, \lambda\in E^* \text{ and } r(\gamma)=r(\lambda)\in H\},
\end{equation}
see \cite[Lemma~2.4.1]{book}.

\begin{Remark}\label{papaterra} Notice that if $J$ is an ideal of $L_K(E)$ then $H(J)$ is hereditary and saturated, see \cite[Lemma~2.4.3]{book} and \cite[Proposition[4.7]{dd2}. 
\end{Remark}
 
In our work, we will use the following characterization of the ideals in $L_K(E)$.

\begin{Proposition}\cite[Theorem 4]{ranga1}\label{propranga} Let $E$ be an arbitrary graph and $I$ a nonzero ideal of $L_K(E)$. Let $H=H(I)$ and define $S=\{v\in B_H: v^H\in I\}$. Then $I$ is generated by $H\cup \{v^H:v\in S\}\cup Y$, where Y is a set of mutually orthogonal elements of the form

$u+\sum\limits_{i=1}^nk_ig^{r_i}$ in which the following statements hold:
\begin{enumerate}
    \item $g$ is a cycle with no exits in $E^0\setminus H$ based at a vertex $u \in E^0\setminus H$;
    \item $k_i\in K$ with at least one $k_i\neq 0$, and $r_i$ are positive integers. 
\end{enumerate}
If $I$ is nongraded, then $Y$ is nonempty.
\end{Proposition}

\subsection{$\Z$-graded ideals of $L_K(E)$}

\begin{Definition}
Let $R$ be a ring. We say that $R$ is $\Z$-graded if there is a collection of additive subgroups $\{ R_n \}_{n \in \Z}$ of $R$ with the following two properties.
\begin{enumerate}
\item $R = \bigoplus_{n \in \Z} R_n$.
\item $R_m R_n \subseteq R_{m+n}$ for all $m,n \in \Z$.
\end{enumerate}
The subgroup $R_n$ is called the \emph{homogeneous component of $R$ of degree $n$}, and the elements of $R_n$ are said to be \emph{homogeneous of degree $n$}.
\end{Definition}

If $E$ is a graph, then we may define a $\Z$-grading on the associated Leavitt path algebra $L_K(E)$ by setting, for each $n\in \Z$, \[L_K(E)_n := \text{span}_K \left\{ \alpha \beta^* : \alpha, \beta \in E^* \text{ and } |\alpha | - | \beta| = n  \right\}.\]

We recall the definition of a graded ideal.

\begin{Definition}
If $R$ is a $\Z$-graded ring, then a subspace (ideal) $I$ of $R$ is a \emph{$\Z$-graded subspace (ideal)} (or just \emph{graded}) if $I = \bigoplus_{n \in \Z} (I \cap R_n)$, or equivalently, if $y=\sum y_n \in I$ with each $y_n\in R_n$, then $y_n \in I$ for all $n$.  
\end{Definition}

The following result gives a precise description of the graded ideals of a Leavitt path algebra in terms of subsets of $E^0$.

\begin{Theorem} {\cite[Theorem 2.5.9]{book}}\label{peixevoador}
Let $E$ be a row-finite graph. Then the map $J\mapsto H(J)$ is a lattice isomorphism between the graded ideals of $L_K(E)$ and the hereditary saturated subsets of $E^0$, with inverse $H\mapsto I(H)$.

\end{Theorem}

From the above theorem, we get the following characterization of $I(H(J))$.

\begin{Corollary}\label{peixeastronauta}\cite[Lemma 2.8.9]{book}
Let $E$ be a graph and $J$ an ideal of $L_K(E)$. Then $I(H(J))$ is the largest graded ideal contained in $J$.
\end{Corollary}


\begin{Definition}\cite[Definition~2.4.11]{book} (The quotient graph by a hereditary subset) Let $E$ be an arbitrary graph, and let $H$ be a hereditary subset of $E^0$. We denote by $E/H$ the quotient graph of $E$ by $H$, defined as follows:
\[(E/H)^0 = E^0 \setminus H \text{ and } (E/H)^1=\{e\in E^1:r(e)\notin H\}.\]
The range and source functions for $E/H$ are defined by restricting the range and source functions of $E$ to $(E/H)$
\end{Definition}

For any 
graph $E$ and ideal $J$ in $L_K(E)$, we 
get the subgraph $E/H(J)$ (notice that $H(J)$ is hereditary, by Remark \ref{papaterra}). We then have the following.

\begin{Proposition}\cite[Corollary~2.4.13]{book}\label{urso}
Let $J$ be a graded ideal of $L_K(E)$.
Then $L_K(E)/J \simeq L_K(E/H(J))$. 
\end{Proposition}

The following result can be used to produce quotient graphs without Condition~(L).

\begin{Proposition}\label{platapus}
Let $J$ be an ideal in $L_K(E)$.
If the graph $E/H(J)$ satisfies Condition~(L), then $J$ is graded.
\end{Proposition}

\begin{proof}

The proof of this proposition is the same as the proof of \cite[Proposition~2.6]{galera} and so we refrain from presenting it here. \end{proof}

\section{Regular ideals and quotients}

Our first step in this section is to define regular ideals. Since the analytical notion does not depend on the norm, there is no change in the definition, which we recall below (as defined in \cite{Hamana}, see also \cite{galera}).

Let $A$ be an algebra.
For a subset $X \subseteq A$, we define $X^\perp$ to be the set
$$ X^\perp = \{a \in A \colon ax= xa = 0 \text{ for all }x\in X\}. $$

\begin{Definition}\label{def: reg} (cf. \cite[Definition~3.1]{CCL})
We call an ideal $J \subseteq A$ a \emph{regular ideal} if $J = J^{\perp\perp}$.
\end{Definition}

Note that, as in the C*-algebra case, if $J$ is an ideal in $A$, then so is $J^\perp$. 
It also always holds that $J \subseteq (J^\perp)^\perp$ and $J^\perp = J^{\perp\perp\perp}$. So $J^\perp$ is always regular.
Next, we show that for a graded ideal $J$ of a $\Z$-graded algebra, $J^\perp$ is also graded.

\begin{Lemma}\label{horajantar}
Let $J$ be a graded ideal of a $\Z$-graded algebra. Then $J^\perp$ is a graded regular ideal.
\end{Lemma}

\begin{proof}
Let $z = \sum\limits_{n\in\Z} z_n \in J^\perp$. We need to prove that each $z_n \in J^\perp$. Let $x\in J$. Since $J$ is graded then $x=\sum\limits x_n$, where each $x_n$ is homogeneous and $x_n\in J$. So it is enough to check that $z_n x_i = x_i z_n = 0$ for each $i,n$.

Since $z\in J^\perp$, and $x_i\in J$ for each fixed $i$, we have that \[\sum_{n\in\Z}  x_i z_n = x_i z =   0= z x_i=\sum_{n\in\Z} z_n x_i,\] and hence the $\Z$-grading of the algebra (and the fact that $x_i$ is homogeneous) implies that $z_n x_i = x_i z_n = 0$ as desired.
\end{proof}

In the spirit of Lemma~\ref{horajantar}, we show below that for Leavitt path algebras, $I^\perp$ is always $\Z$-graded, for each ideal $I$.

\begin{Theorem} \label{avai ganhou do vasco}
Let $E$ be an arbitrary graph and let $I$ be an ideal of $L_K(E)$. Then $I^\perp$  is $\Z$-graded.
\end{Theorem}

\begin{proof} The case $I=0$ is trivial. Moreover, from Lemma~\ref{horajantar}, we get that if $I$ is graded then $I^\perp$ is graded. So, we may suppose that $0\neq I$ and that $I$ is non-graded. Let $H=H(I)$, which, by Remark~\ref{papaterra}, is hereditary and saturated. It follows from Proposition~\ref{propranga} that $I$ is generated by $A=H\cup \{v^H:v\in S\}\cup Y$, with $Y$ nonempty. 

Define the subspace $$V=\{z\in L_K(E):za=0 \,\,\forall a\in A\}.$$ We show that $V$ is $\Z$-graded, according to the $\Z$-grading of $L_K(E)$. Let $z\in V$ and decompose it in its homogeneous components, say $z=\sum \limits_{n\in F} z_n$ where $F\subseteq \Z$ is a finite set. Let $a\in H\cup \{v^H:v\in S\}$. Then, $0=za=\sum\limits_{n\in F} z_na$ and, since $a$ has degree 0 and $L_K(E)$ is $\Z-$graded, we obtain that $z_na=0$ for each $n\in F$. Let $u+\sum\limits_i {k_i}g^{r_i}\in Y$ and define $P=\{n\in F:z_nu\neq 0\}$ and $Q=\{n\in F:z_nu=0\}$. Suppose that $P\neq \emptyset$ and let $n_0$ be the minimum of $P$. Since $g$ is a cycle based on the vertex $u$, we have that  $z_n(u+\sum\limits_i {k_i}g^{r_i})=0$ for each $n\in Q$. Hence,  $$0=z(u+\sum\limits_i {k_i}g^{r_i})=\sum\limits_{n\in P}z_n(u+\sum\limits_i {k_i}g^{r_i})=\sum\limits_{n\in P}z_nu+\sum\limits_{n\in P}\sum\limits_i z_n{k_i}g^{r_i}.$$ Since $L_K(E)$ is $\Z$-graded, we obtain that $z_{n_0}u=0$, which is impossible, since $n_0\in P$. Then $P=\emptyset$, and it follows that $z_nu=0$ for each $n\in F$, and so $z_n(u+\sum\limits_i {k_i}g^{r_i})=0$ for each $n\in F$. 
This shows that $z_nx=0$ for each $x\in H\cup \{v^H:v\in S\}\cup Y$, and hence $z_n\in V$ for each $n$. Therefore, $V$ is a graded subspace of $L_K(E)$.

Analogously, one shows that $$V'=\{z\in L_K(E):az=0\,\,\forall a\in A\}$$ is a graded subspace of $L_K(E)$.

Let $$W=\{z\in L_K(E):\,\,zxa=0 \text{ for each } x\in L_K(E) \text{ and }a\in A\}.$$ We show that $W$ is a $\Z$-graded subspace. First, notice that $z\in W$ if, and only if, $zxa=0$ for each homogeneous element $x\in L_K(E)$ and each $a\in A$. Let $z\in W$ and decompose it in homogeneous components, say $z=\sum\limits_{n\in F'} z_n$ where $F'$ is a finite subset of $\Z$. Then, for each  
homogeneous element $x$ of $L_K(E)$, we get that $zx=\sum\limits_{n\in F'} z_nx\in V$. Since $V$ is graded, we obtain that $z_nx\in V$, and so $z_nx a=0$ for each $a\in A$. Therefore, $z_n\in W$ for each $n\in F'$ and hence $W$ is graded. 

Analogously, one proves that $$W'=\{z\in L_K(E):\,\,axz=0 \text{ for each } x\in L_K(E) \text{ and }a\in A\}$$ is graded.

Finally, notice that $$I^{\perp}=V\cap V'\cap W\cap W',$$ which implies, since $V,V',W,W'$ are all graded, that $I^\perp$ is graded. \end{proof}

As a consequence of the previous theorem, we obtain that all regular ideals in $L_K(E)$ are graded. This is different from the analytical setting (see Proposition~3.7 in \cite{galera} and Remark~\ref{ondesol}).

\begin{Corollary}\label{regular are graded} Let $E$ be an arbitrary graph and $I$ a regular ideal of $L_K(E)$. Then $I$ is $\Z$-graded.
\end{Corollary}

\begin{proof}
Let $I$ be an ideal of $L_K(E)$. Since $I^\perp$ is also an ideal, by Theorem~\ref{avai ganhou do vasco}, $I^{\perp\perp}$ is graded. Hence, $I=I^{\perp\perp}$ is graded.
\end{proof}

Our next goal is to describe the vertex set of a regular ideal. For this, we need the following definitions.

\begin{Definition}
For an ideal $I$ of $L_K(E)$, we define $\overline{H}(I)\subseteq
   E^0$ as 
$$ \overline{H}(I) := \{ s(\alpha): \alpha\in E^* \text{ and }  r(\alpha)\in H(I)\}.$$
\end{Definition}

The set $\overline{H}(I)$ should be thought of as some kind of "closure" of the set $H(I)$. In particular notice that it always hold that $H(I)\subseteq \overline{H}(I)$.

\begin{Definition}
Let $E$ be a graph and $v\in E^0$. The tree of $v$, denoted $T(v)$, is the set \[T(v):=\{r(\alpha):\alpha\in E^*, s(\alpha)=v\}.\]
\end{Definition}

\begin{Proposition}\label{tomaraqdc} (cf. \cite[Proposition~3.2]{CCL})
Let $E$ be a row-finite directed graph and $J \subseteq L_K(E)$ be a graded ideal.
Then
\begin{enumerate}
\item $J^\perp = I(E^0\backslash \ol{H}(J))$;
\item $J^{\perp\perp} = I(\{w \in E^0 \colon T(w) \subseteq \ol{H}(J)\})$;
\item $J$ is regular if and only if $H(J) = \{w \in E^0 \colon T(w) \subseteq \ol{H}(J)\}$.
\end{enumerate}
\end{Proposition}

\begin{proof}

By Lemma~\ref{horajantar}, $J^\perp$ is graded.
Hence by Theorem~\ref{peixevoador} $J^\perp = I(H(J^\perp))$.
Thus, to prove (i), it suffices to prove that $H(J^\perp) = E^0\backslash \ol{H}(J)$. This is done similarly to what is done to prove the item (i) in \cite[Proposition~3.4]{galera}. We include the proof for completeness.

First we prove that $H(J^\perp) \subseteq E^0\backslash \ol{H}(J)$. Let $v\in H(J^\perp)$ (so $v\in J^\perp$). If $v\in \overline{H}(J)$ then there exists $\alpha \in E^*$ such that $v=s(\alpha)$ and $r(\alpha) \in H(J)$. But then $\alpha \in J$, since $I(H(J))=J$ (by Theorem~\ref{peixevoador}) and $\alpha$ belongs to the ideal generated by $H(J)$. Hence $v \alpha = \alpha \neq 0$, a contradiction, since $v\in J^\perp$ and $\alpha\in J$. 

For the converse inclusion, let $w\notin \ol{H}(J)$. By definition, for each $\alpha \in E^*$ with $r(\alpha)\in H(J)$ we have $s(\alpha)\in \ol{H}(J)$. Hence, for such $\alpha$, $w s(\alpha)=0$ and therefore $w \alpha=0=\alpha^* w $. Using again that $I(H(J))=J$ we obtain, by (\ref{jacare}), that $$ J :=\spn\{\gamma \lambda^*: \gamma, \lambda\in E^* \text{ and } r(\gamma)=r(\lambda)\in H(J)\}.$$
So  $w\in J^\perp$ and hence $w\in H(J^\perp)$.

Next we prove (ii). Since $J^{\perp\perp} = I (H(J^{\perp\perp}))$, it is enough to prove that $H(J^{\perp\perp}) = \{w \in E^0 \colon T(w) \subseteq \ol{H}(J)\}$.

Take $w\in E^0 $ with $z\in T(w)$ such that $z \notin \ol{H}(J)$. By (i) $z\in J^\perp$. Let $\alpha \in E^*$ be such that $s(\alpha)=w$ and $r(\alpha)=z$. Then $\alpha z \in J^\perp$ and $w \alpha z = \alpha \neq 0$. Hence $w \notin J^{\perp\perp}$.

On the other hand, if $u\notin H(J^{\perp\perp})$ then, by (i) and (\ref{jacare}), there exists $\gamma, \lambda \in E^*$ with $r(\gamma)=r(\lambda) \in E^0\backslash \ol{H}(J) $ and such that $u \gamma \lambda^*\neq 0$ or $\gamma \lambda^* u \neq 0$. Assume that $u \gamma \lambda^*\neq 0$ and $|\gamma|>0$ (the case $|\gamma|=0$ is direct). Then $s(\gamma) = u$ and hence $r(\gamma) \in T(u)$, what implies that $u\notin \{w \in E^0 \colon T(w) \subseteq \ol{H}(J)\}$. The case $\gamma \lambda^* u \neq 0$ is dealt with analogously.

We now prove (iii). As proved above, $H(J^{\perp\perp}) = \{w \in E^0 \colon T(w) \subseteq \ol{H}(J)\}$ and so, by Remark~\ref{papaterra}, $\{w \in E^0 \colon T(w) \subseteq \ol{H}(J)\}$ is hereditary and saturated. Suppose first that $J$ is regular. Then 
item (ii) and Theorem~\ref{peixevoador} imply that \[\begin{array}{ll} H(J) & = H(J^{\perp \perp}) = H(I(\{w \in E^0 \colon T(w) \subseteq \ol{H}(J)\}))\\ &
= \{w \in E^0 \colon T(w) \subseteq \ol{H}(J)\}.
\end{array}\]

 For the converse, use again (ii) and the fact that $J$ is graded (along with Theorem~\ref{peixevoador}) to get that $J = I(H(J))=J^{\perp\perp}$.
\end{proof}

Next, we will explore relations between regular graded ideals and cycles without exits in row-finite graphs.

Let $E$ be a graph and $P_c(E)$ be the set of all the vertices in cycles without exit in $E$. Fix $H$ a saturated and hereditary subset of $E^0$. Notice that if $c$ is a cycle in $E$, then some vertex of $c$ belongs to $H$ if, and only if, each vertex in $c$ belongs to $H$. So, it is possible to define the map
$f:P_c(E)\setminus H\rightarrow P_c(E/H)$ by $f(v)=v$. Clearly, $f$ is injective.

\begin{Proposition}\label{bijecao}
Let $E$ be a row-finite graph and  let $H\subseteq E^0$ be saturated and hereditary. If $I(H)$ is regular, then  $f:P_c(E)\setminus H\rightarrow P_c(E/H)$ defined as above is a bijective map. 
\end{Proposition} 

\begin{proof}
All we have to do is to show that $f$ is surjective. Let $v\in P_c(E/H)$. Then $v$ is a vertex of a cycle without exit $c=e_1...e_n$ in $E/H$. Notice that $c$ is also a cycle in $E$. We show that $v\in P_c(E)\setminus H$.

 Suppose that $c$ has an exit in $E$, that is, there exists some edge $e\in E$ such that $e\neq e_i$ for some $i$ and $s(e_i)=s(e)$. Then,  $r(e)\in H$ because, otherwise, $c$ has an exit in $E/H$. Since $H$ is hereditary, we have that $r(e\beta)\in H (\subseteq \overline{H}(I(H)))$ for each finite path $\beta$ with $s(\beta)=r(e)$.

Let $u$ be a vertex in $c$. Then, there exist a finite path $\alpha$ such that $s(\alpha)=u$ and $r(\alpha)=r(e)$. Since $r(e)\in H$, we obtain that $u\in \overline{H}(I(H))$. 

We conclude that, for each path $\gamma$ with $s(\gamma)=v$, $r(\gamma)\in \overline{H}(I(H))$. This means that $T(v)\subseteq \overline{H}(I(H))$. Hence, from the second item of Proposition~\ref{tomaraqdc}, we obtain that $v\in I(H)^{\perp\perp}$. Since $I(H)$ is regular, we get that  $v\in I(H)$ and hence $v\in H$, which is impossible, since $v$ is a vertex of a cycle in $E/H$. Therefore, $v\in P_c(E)\setminus H$ and $f$ is surjective.   
\end{proof}

From the previous proposition, we get the following consequences.

\begin{Corollary}\label{cor1}
Let $E$ be row-finite directed graph and $H$ be a saturated and hereditary subset of $E^0$.
\begin{enumerate}
    \item If $E/H$ satisfies condition $(L)$ then $P_c(E)\subseteq H$.
    \item If $I(H)$ is regular then $E/H$ satisfies condition $(L)$ if, and only if, $P_c(E)\subseteq H$.
\end{enumerate}
\end{Corollary}

Next, we show that, as in the C*-algebraic setting (see \cite[Corollary~3.8]{galera}), quotients by regular ideals preserve Condition~(L).

\begin{Proposition}\label{lagosta}
Let $E$ be a row-finite graph satisfying Condition~(L).
Let $J$ be a regular ideal in $L_K(E)$.
Then $E/H(J)$ satisfies Condition~(L) and $L_K(E)/J \simeq L_K(E/H(J))$.
\end{Proposition}

\begin{proof}
Since $E$ satisfies condition $(L)$, we have that $P_c(E)=\emptyset$. Hence, $P_c(E)\subseteq H(J)$. Since $J$ is regular, by Corollary~\ref{regular are graded}, $J$ is graded. So, $I(H(J))=J$ (see Theorem~\ref{peixevoador}) is regular and, from the second item of Corollary~\ref{cor1}, we obtain that $E/H(J)$ satisfies condition $(L)$.

The last statement of the proposition follows from  Corollary~\ref{regular are graded} and Proposition~\ref{urso}.
\end{proof}

\begin{Remark}
The hypothesis that $J$ is regular can not be dropped in the previous proposition. For instance, let $E$ be the graph with two edges $f,g$ and two vertices $u,v$, with $s(f)=r(f)=u=s(g)$ and $r(g)=v$. Let $H=\{v\}$, which is saturated and hereditary, and let $J=I(H)$, which is a graded ideal. Notice that $J$ is not regular. Indeed, since $P_c(E)=\emptyset$, if $J$ is regular then, by the second item of Corollary~\ref{cor1}, $E/H$ satisfy Condition~(L), which is not the case. So, $E$ is a graph which satisfies condition $(L)$, and $J$ is a graded ideal (which is not regular) such that $E/H(J)=E/H$ does not satisfy Condition~ (L).
\end{Remark}

We refer the reader to \cite[Example~3.9]{galera} for an example which shows that not all ideals $J$ with $E/J$ satisfying Condition~(L) are regular.

\begin{Remark}\label{ondesol}
The theory of regular ideals has divergences in the analytical and algebraic settings. A good example to consider is the graph $E$ with a single vertex and a single edge. In this case $C^*(E) = C(\mathbb{T})$ (the C*-algebra of the continuous complex functions with domain $\mathbb{T}=\{z\in \mathbb{C}:|z|=1\}$). Let $X=\{a+ib\in \mathbb{T}:a\geq 0 \},$ and $I=\{f\in C(\mathbb{T}): f_{|_X=0}\}$. Then $I$ is a regular ideal of $C(\mathbb{T})$ but is not gauge-invariant (see \cite{illinois} for a description of gauge-invariant ideals in graph C*-algebras).
In the algebraic framework, we have that $L_K(E) = K[x,x^{-1}]$, the Laurent polynomial algebra and, by Corollary~\ref{regular are graded}, each regular ideal is graded.

So, while there are many non-trivial regular, non-gauge-invariant ideals in $C^*(E)$, in the algebraic settings all the regular ideals are graded. Furthermore, while $L_K(E)$ does not contain any non-trivial graded ideal, it contains many non-graded ideals (see \cite[Remark~2.1.6]{book}).
\end{Remark}

We finish the paper noticing that maximal ideals $I$ such that $I^\perp\neq \{0\}$ are a good source of regular ideals, see Proposition~\ref{alface} below. A description of the existence of maximal ideals in Leavitt path algebras is given in \cite{Muge}. Furthermore, in some situations every prime ideal of a Leavitt path algebra is maximal, see \cite[Theorem 6.1]{ranga}.

\begin{Proposition}\label{alface}
Let $A$ be an algebra with local units (which is the case for Leavitt path algebras). If $I$ is a maximal ideal of $A$ then $I$ is regular or $I^\perp = \{0\}$.
\end{Proposition}
\begin{proof}
Suppose that $I$ is a non-zero maximal ideal. Since $I\subseteq I^{\perp \perp}$, we have that either $I=I^{\perp \perp}$ or $I^{\perp \perp}=A$ in which case $I^\perp =I^{\perp\perp\perp}= A^\perp=\{0\}$ (since $A$ has local units). 
\end{proof}

\begin{Remark}\label{local units} If $E$ is an arbitrary graph and $I$ is a maximal ideal of $L_K(E)$ such that $I^{\perp}\neq \{0\}$ then, $I$ is graded.
\end{Remark}

\section{Acknowledgments}

The authors are indebt with the anonymous referee, who pointed out many improvements in the paper and, in particular, has pointed the path to the proof of Theorem~\ref{avai ganhou do vasco}. We would also like to thank Prof. Pere Ara for his insightful suggestions regarding the manuscript.


\begin{thebibliography}{99}

\bibitem{Gene} G. Abrams, 
\textit{Leavitt path algebras: the first decade}. Bulletin of Mathematical Sciences. {\bf 5} (2015), no. 1, 59--120.

\bibitem{book} G. Abrams, P. Ara, M. Siles Molina, Leavitt path algebras. \emph{Lecture Notes in Mathematics. Springer} (2017). 21--91.


\bibitem{ABHS} P. Ara, J. Bosa, R. Hazrat and A. Sims, \textit{Reconstruction of graded groupoids from graded Steinberg algebras}, Forum Math., 29(5) (2017), 1023-1037.


\bibitem{AHHS} P. Ara, R. Hazrat, H. Li, A. Sims,
\textit{Graded Steinberg algebras and their representations}. Algebra Number Theory 12 (1) (2018), 131--172.

\bibitem{Pardo} P. Ara, M. Moreno, E. Pardo, \textit{Nonstable K-theory for Graph Algebras}, Algebr. Represent. Theory 10 (2007), 157--178.

 \bibitem{AMMS} G. Aranda Pino, D. Mart\'{\i}n Barquero, C. Mart\'{\i}n Gonz\'alez, M. Siles Molina, \textit{Socle theory for Leavitt path algebras of arbitrary graphs}. Rev. Mat. Iberoam.  26  (2010),  no. 2, 611--638. 
 
 \bibitem{illinois} T. Bates, J. H. Hong, I. Raeburn, W. Szymanski, \textit{The ideal structure of the {C}*-algebras of infinite graphs}, Illinois J. Math. 46
Volume 46 (2002), no. 4, 1159--1176.
 
 
 \bibitem{galera} J. H. Brown, A H. Fuller, D. R. Pitts, S. A. Reznikoff, \textit{Regular ideals of graph algebras}, (2020)	arXiv:2006.00395 [math.OA].
 
 \bibitem{CCL} C. G. Canto, D. Mart\'{\i}n Barquero, C. Mart\'{\i}n Gonz\'alez, \textit{Invariant ideals in Leavitt path algebras}, (2020) arXiv:2006.12876 [math.RA].
 
 \bibitem{Canto} C. G. Canto, D. Gon\c{c}alves, \textit{Representations of relative Cohn path algebras},  J. Pure Appl. Algebra 224 (2020), 106310.
 
 
\bibitem{CMMS} {L. O. Clark, D. Mart\'{\i}n Barquero, C. Mart\'{\i}n Gonz\'{a}lez, M. Siles Molina.} \textit{Using the Steinberg algebra model to determine the center of any Leavitt path algebra.} Isr. J. Math. 230 (2019), 23--44.

\bibitem{CHR} L. O.  Clark, R. Hazrat and S. W. Rigby, \textit{Strongly graded groupoids and strongly graded Steinberg algebras}, J. Algebra 530 (2019), 34--68.

\bibitem{dd2} D. Gon\c{c}alves, D. Royer, \textit{Leavitt path algebras as partial skew group rings}, Comm. Algebra 42 (2014), 127--143.

\bibitem{dy} D. Gon\c{c}alves, G. Yoneda,
\textit{Free path groupoid grading on Leavitt path algebras}, Internat. J. Algebra Comput. 26 (2016), 1217--1235.

\bibitem{Muge} S. Esin and M. Kanuni, \textit{Existence of maximal ideals in {L}eavitt path algebras}, Turk. J. Math. 42 (2018), 2081--2090

\bibitem{Hamana} M. Hamana, \emph{The centre of the regular monotone completion of a
  {$C^{\ast} $}-algebra}, J. London Math. Soc. (2) \textbf{26} (3) (1982),  522--530. 
  

\bibitem{HR} R. Hazrat, K. M. Rangaswamy, \textit{On graded irreducible representations of Leavitt path algebras}, J. Algebra 450 (2016), 458--486.

\bibitem{HRS} R. Hazrat, K. M. Rangaswamy, A. K. Srivastava, \textit{Structure theory of graded regular graded self-injective rings and applications} (2018), arXiv:1808.03905 [math.RA]

 \bibitem{ranga} K.M. Rangaswamy, \textit{The theory of prime ideals of Leavitt path algebras over
arbitrary graphs}, J. Algebra 375 (2013), 73--96.

\bibitem{ranga1} K.M. Rangaswamy, \textit{Generators of Two-Sided Ideals of Leavitt
Path Algebras over Arbitrary Graphs}, Comm. Algebra 42 (2014), 2859-–2868.

\bibitem{Ri} S. W. Rigby, \textit{The Groupoid Approach to Leavitt Path Algebras.} In: Ambily A., Hazrat R., Sury B. (eds) Leavitt Path Algebras and Classical K-Theory. Indian Statistical Institute Series.  Springer, Singapore (2020).



\bibitem{Tomforde} M. Tomforde, \textit{
Uniqueness theorems and ideal structure for Leavitt path algebras}, J. Algebra 318(1) (2007), 270--299.








 
 










\end{thebibliography}
\end{document}